\documentclass[12pt,reqno]{amsart}

\usepackage{amsthm, amsmath, amsfonts, amssymb, graphicx, tikz-cd, makecell, hyperref, mathtools} 
\DeclareMathAlphabet{\mathbbb}{U}{bbold}{m}{n}
\usepackage[capitalize,noabbrev]{cleveref} 
\usepackage[shortlabels]{enumitem}
\usetikzlibrary{arrows.meta,calc}
\usepackage{needspace}

\newtheorem{theorem}{Theorem}[section]
\newtheorem{lemma}[theorem]{Lemma}
\newtheorem{proposition}[theorem]{Proposition}
\newtheorem{corollary}[theorem]{Corollary}

\theoremstyle{definition}
\newtheorem{definition}[theorem]{Definition}

\newtheorem{convention}[theorem]{Convention}

\newtheorem{remark}[theorem]{Remark}
\newtheorem{claim}[theorem]{Claim}

\crefname{theorem}{Theorem}{Theorems}
\crefname{lemma}{Lemma}{Lemmas}
\crefname{proposition}{Proposition}{Propositions}
\crefname{corollary}{Corollary}{Corollaries}
\crefname{conjecture}{Conjecture}{Conjectures}
\crefname{definition}{Definition}{Definitions}
\crefname{example}{Example}{Examples}
\crefname{convention}{Convention}{Conventions}
\crefname{question}{Question}{Questions}
\crefname{remark}{Remark}{Remarks}
\crefname{claim}{Claim}{Claims}

\AddToHook{env/theorem/begin}{\crefalias{theorem}{theorem}}
\AddToHook{env/lemma/begin}{\crefalias{theorem}{lemma}}
\AddToHook{env/proposition/begin}{\crefalias{theorem}{proposition}}
\AddToHook{env/corollary/begin}{\crefalias{theorem}{corollary}}
\AddToHook{env/conjecture/begin}{\crefalias{theorem}{conjecture}}
\AddToHook{env/definition/begin}{\crefalias{theorem}{definition}}
\AddToHook{env/example/begin}{\crefalias{theorem}{example}}
\AddToHook{env/convention/begin}{\crefalias{theorem}{convention}}
\AddToHook{env/question/begin}{\crefalias{theorem}{question}}
\AddToHook{env/remark/begin}{\crefalias{theorem}{remark}}
\AddToHook{env/claim/begin}{\crefalias{theorem}{claim}}

\newcounter{mycount}

\newcommand{\myref}[1]{\hyperref[#1]{#1}}

\newcommand{\sfour}{\mathsf{S4}}
\newcommand{\ipc}{\mathsf{IPC}}

\newcommand{\M}{\mathsf{M}}
\newcommand{\N}{\mathsf{N}}

\renewcommand{\L}{\mathsf{L}}
\newcommand{\F}{\mathfrak{F}}
\newcommand{\G}{\mathfrak{G}}
\newcommand{\Kur}{\mathsf{Kur}}
\newcommand{\kur}{\mathsf{kur}}

\newcommand{\mipc}{\mathsf{MIPC}}
\newcommand{\mgrz}{\mathsf{MGrz}}

\newcommand{\msfour}{\mathsf{MS4}}

\newcommand{\sfive}{\mathsf{S5}}

\newcommand{\ms}{\mathsf{MS4}}

\newcommand{\qsfour}{\mathsf{QS4}}

\newcommand{\cpc}{\mathsf{CPC}}

\renewcommand{\L}{\mathsf{L}}
\newcommand{\Lae}{\mathcal{L_{\forall\exists}}}

\newcommand{\iqc}{\mathsf{IQC}}
\newcommand{\sk}{\rho \hspace{0.08em}}

\newcommand{\msfrm}{\mathbf{DF}_{\ms}}
\newcommand{\mipcfrm}{\mathbf{DF}_{\mipc}}
\newcommand{\Grz}{\mathsf{Grz}}
\newcommand{\grz}{\mathsf{grz}}
\newcommand{\QGrz}{\mathsf{QGrz}}

\newcommand{\sfrm}{\mathbf{DF}_{\sfour}}
\newcommand{\ipcfrm}{\mathbf{DF}_{\ipc}}

\newcommand{\GKur}{\mathsf{GKur}}
\newcommand{\LKur}{\mathsf{LKur}}
\newcommand{\Rnew}{\overline{R}}
\newcommand{\Enew}{\overline{E}}
\newcommand{\Qnew}{Q_{\Enew}}

\DeclareMathOperator{\qmax}{qmax}
\newcommand{\bbox}{\blacksquare}
\newcommand{\bdia}{\blacklozenge}
\newcommand{\KP}{\mathsf{KP}}
\newcommand{\GKP}{\mathsf{GKP}}
\newcommand{\LKP}{\mathsf{LKP}}

\newcommand{\Q}{Q_E}

\pgfdeclarelayer{edgelayer}
\pgfdeclarelayer{nodelayer}
\pgfsetlayers{edgelayer,nodelayer,main}

\tikzstyle{none}=[inner sep=0pt]
\tikzstyle{black dot}=[fill=black, draw=black, shape=circle, inner sep=0, minimum size=3.5pt]

\tikzstyle{to}=[->]
\tikzstyle{mapsto}=[{|->}]
\tikzstyle{none dashed}=[-, dashed]
\tikzstyle{dashed to}=[->, dashed]
\tikzstyle{dashed mapsto}=[{|->}, dashed]
\tikzstyle{Latex arrow}=[{-{Latex[width=1mm]}}]
\tikzstyle{dashed Latex arrow}=[{-{Latex[width=1mm]}}, dashed]

\newcommand\clusterone[2]{
  \path[draw,red] let \p1=(#1)
    in \pgfextra{
    \pgfmathsetmacro{\radius}{#2*0.3}
  }
  (\p1) circle(\radius cm);
}

\newcommand\clustertwo[4]{
  \path[draw,red] let \p1=(#1), \p2=(#2), \p3=($(\p1)!.5!(\p2)$)
  in \pgfextra{
    \pgfmathsetmacro{\angle}{atan2(\y2-\y1,\x2-\x1)}
    \pgfmathsetmacro{\focal}{veclen(\x2-\x1,\y2-\y1)/2/1cm}
    \pgfmathsetmacro{\lentotcm}{\focal*2*#3}
    \pgfmathsetmacro{\axeone}{(\lentotcm - 2 * \focal)/2+\focal}
    \pgfmathsetmacro{\axetwo}{sqrt((\lentotcm/2)*(\lentotcm/2)-\focal*\focal}
    \pgfmathsetmacro{\newaxetwo}{\axetwo*0.5*#4}
  }
  (\p3) ellipse[x radius=\axeone cm,y radius=\newaxetwo cm, rotate=\angle];
}

\setlist[enumerate,1]{label={\upshape(\arabic*)},ref=\arabic*}

\makeatletter 
\edef\plabelformat{(\string#2\string#1\string#3)}
\edef\plabelrangeformat{(\string#3\string#1,\string#2\string#6)}
\newcommand{\plabel}[1]{\label{#1}
\immediate\write\@auxout{\noexpand\crefformat{#1}{\noexpand\cref{#1}\plabelformat}
\noexpand\crefmultiformat{#1}{\noexpand\cref{#1}\plabelformat}{,\plabelformat}{,\plabelformat}{,\plabelformat}
\noexpand\crefrangeformat{#1}{\noexpand\cref{#1}\plabelrangeformat}}}
\makeatother

\makeatletter
\@namedef{subjclassname@2020}{\textup{2020} Mathematics Subject Classification}
\makeatother

\begin{document}

\title{Failure of Esakia's theorem in the monadic setting}

\author{G.~Bezhanishvili}
\address{New Mexico State University}
\email{guram@nmsu.edu}

\author{L.~Carai}
\address{University of Milan}
\email{luca.carai.uni@gmail.com}


\begin{abstract}
Esakia's theorem states that Grzegorczyk's logic 
is the greatest 
modal companion of intuitionistic propositional calculus. 
We prove that already the one-variable fragment of intuitionistic predicate calculus 
does not have a greatest modal companion, yielding that Esakia's theorem fails in the monadic setting. 
\end{abstract}

\maketitle
\tableofcontents

\section{Introduction}

In~\cite{God33} G\"odel suggested to interpret intuitionistic logic as a fragment of modal logic. McKinsey and Tarski~\cite{MT48} proved that the G\"odel translation embeds the intuitionistic propositional calculus $\ipc$ faithfully into the modal logic $\sfour$. There are many other modal logics above $\sfour$ into which $\ipc$ embeds faithfully, known as modal companions of $\ipc$ (see, e.g., \cite[Sec.~9.6]{CZ97}). It is a well-known result of Esakia~\cite{Esa79} that the Grzegorczyk logic $\Grz$ is the greatest modal companion of $\ipc$.

The situation becomes more complicated in the predicate case. Let $\iqc$ be the intuitionistic predicate calculus, $\qsfour$ the predicate $\sfour$, and $\QGrz$ the predicate Grzegorczyk logic. It is a well-known result of Rasiowa and Sikorski \cite{RS53} that the G\"odel translation embeds $\iqc$ faithfully into $\qsfour$. In \cite{Pan89} it is claimed that it also embeds $\iqc$ faithfully into $\QGrz$, and in \cite{Nau91} it is claimed that $\QGrz$ is no longer the greatest modal companion of $\iqc$ (see also \cite[Thm.~2.11.14]{GSS09}).
However, the proofs in \cite{Pan89,Nau91} use the Flagg-Friedman translation \cite{FF86} of $\qsfour$ to $\iqc$, which Inou\' e \cite{Ino92} showed is 
not faithful.\footnote{We thank Valentin Shehtman for drawing our attention to \cite{Ino92}.}
In fact, as we will see in Theorem~\ref{thm:Naumov}, the extension of $\QGrz$ considered by Naumov is not a modal companion of $\iqc$.
Nevertheless, as we will see in Theorem~\ref{thm:no greatest mod comp},
the monadic fragment of $\iqc$, where we only consider formulas with one fixed variable, does not have a greatest modal companion.
The full predicate version requires further examination (see Remark~\ref{rem:discussion predicate}).

The study of the monadic fragment of classical predicate calculus was initiated by Hilbert and Ackermann \cite{HA28}, 
and Wajsberg \cite{Waj33} proved that $\mathsf{S5}$ axiomatizes this fragment.
Prior~\cite{Pri57} introduced the monadic intuitionistic calculus $\mipc$, and  Bull~\cite{Bul66} proved that it axiomatizes the monadic fragment of $\iqc$. 
The monadic fragment of $\qsfour$ is axiomatized by $\ms$ and that of $\QGrz$ by $\mgrz$ (see \cite{FS77,Esa88,BK24}).
The G\"odel translation embeds $\mipc$ faithfully into both $\ms$ and $\mgrz$, but our main result shows that there is no greatest modal companion of $\mipc$, thus yielding a failure of Esakia's theorem in the monadic setting.

We achieve this by introducing modal versions of the monadic Kuroda formula. The Kuroda formula $\forall x\neg\neg P(x)\to\neg\neg\forall x P(x)$ is not provable in $\iqc$. Heyting considered this as one of the most striking features of $\iqc$ (see \cite[p.~108]{Hey56}). 
We introduce two natural modal versions of the monadic Kuroda formula $\forall\neg\neg p\to\neg\neg\forall p$, 
which result in two extensions of $\ms$ that we term the global and local Kuroda logics. 
The global Kuroda logic $\GKur$ is obtained by adding the G\"odel translation of $\forall\neg\neg p\to\neg\neg\forall p$ to $\ms$,
while the local Kuroda logic $\LKur$ is an appropriate weakening of $\GKur$.  
We prove that $\GKur$ is a modal companion of $\mipc+\forall\neg\neg p\to\neg\neg\forall p$, that $\LKur$ is a modal companion of $\mipc$, and that $\LKur$ is not comparable with $\mgrz$. In addition,  $\mgrz\vee\LKur=\mgrz\vee\GKur$, thus yielding that $\mipc$ can't have the greatest modal companion.  

We conclude the introduction by providing a brief semantic explanation of the validity of Esakia's theorem in the propositional case and why it fails in the monadic case. An adequate semantics for $\ipc$ is given by the category $\ipcfrm$ of descriptive intuitionistic frames and that
for $\sfour$ by the category $\sfrm$ of descriptive $\sfour$-frames (see, e.g., \cite[Thm.~8.36]{CZ97}). There is an embedding $\sigma \colon \ipcfrm \to \sfrm$, which has a left adjoint $\rho \colon \sfrm \to \ipcfrm$ such that $\F \cong \rho\sigma\F$ for each $\F \in \ipcfrm$ (see, e.g., \cite[Secs.~8.3 and 8.4]{CZ97} or \cite[Secs.~2.2 and 2.5]{Esa19} for the algebraic formulation). If $\M$ is a modal companion of $\ipc$, then for each $\F \in \ipcfrm$ there is $\G \in \sfrm$ such that $\G \vDash \M$ and $\F \cong \rho\G$. 
This implies that $\sigma\F \vDash \M$. But $\Grz$ is the logic of $\{ \sigma\F : \F \in \ipcfrm\}$. Thus, $\M \subseteq \Grz$, yielding Esakia's theorem.

In the monadic case, an adequate semantics for $\mipc$ is given by the category $\mipcfrm$ of descriptive $\mipc$-frames and that for $\msfour$ by the category $\msfrm$ of descriptive $\msfour$-frames. We still have a functor $\rho \colon \msfrm \to \mipcfrm$ (see Theorem~\ref{thm:skeleton descriptive}\eqref{thm:skeleton descriptive:item1a}). However, the analogue of $\sigma$ is no longer well defined, and it remains open whether $\rho$ has a right adjoint. It is this lack of a nice semantic correspondence in the monadic case that is responsible for the failure of Esakia's theorem (as well as that of the Blok-Esakia theorem \cite{BC24a}). We further discuss this in the conclusions, after establishing our main results.

\section{Preliminaries}

In this section we briefly recall $\mipc$, $\ms$, $\mgrz$, and their semantics. Let $\Lae$ be a propositional bimodal language, where the modalities are denoted by $\forall$ and $\exists$.

\begin{definition} \label{def: MIPC}
\hfill\begin{enumerate}
\item The \textit{monadic intuitionistic propositional calculus} $\mipc$ is the smallest set of formulas of $\Lae$ containing
\begin{enumerate}
\item all theorems of $\ipc$;
\item $\sfour$-axioms for $\forall$: $\forall(p\land q)\leftrightarrow(\forall p\land\forall q)$, $\forall p \rightarrow p$, $\forall p \rightarrow \forall \forall p$;
\item $\sfive$-axioms for $\exists$: $\exists(p\vee q)\leftrightarrow(\exists p\vee\exists q)$, $p \rightarrow \exists p$, $\exists \exists p \rightarrow \exists p$, $(\exists p \land \exists q) \rightarrow \exists (\exists p \land q)$; 
\item connecting axioms: $\exists\forall p\rightarrow\forall p$, $\exists p \rightarrow \forall\exists p$;
\end{enumerate}
and closed under modus ponens, substitution, and necessitation $(\varphi / \forall \varphi )$.
\item An \emph{extension of $\mipc$} is a set of formulas of $\Lae$ containing $\mipc$ and closed under the above rules of inference. 
\end{enumerate}
\end{definition}

Kripke semantics for $\mipc$ was introduced in \cite{Ono77} (see also \cite{FS78b}).
We recall that a quasi-order is a reflexive and transitive binary relation. Given a quasi-order $Q$ on a set $X$, 
we recall that $U \subseteq X$ is a \textit{$Q$-upset} if $x\in U$ and $xQy$ imply $y\in U$.

\begin{convention}\label{conv:EQ}
For a quasi-order $Q$ on a set $X$, we denote by $E_Q$ the equivalence relation defined by 
\begin{equation*}
x E_Q y \iff xQy \text{ and } yQx.  
\end{equation*}
\end{convention}

\begin{definition}\label{def:mipc-frame}
An \textit{$\mipc$-frame} is a tuple $\F=(X,R,Q)$ such that
\begin{enumerate}
\item\label{def:mipc-frame:item1} $R$ is a partial order on $X$;
\item\label{def:mipc-frame:item2} $Q$ is a quasi-order on $X$;
\item\label{def:mipc-frame:item3} $R \subseteq Q$;
\item\label{def:mipc-frame:item4} $x Q y \Longrightarrow \exists z \in X : x R z \; \& \; z E_Q y$.
\end{enumerate}
\begin{figure}[!ht]
\begin{center}
\begin{tikzpicture}
	\begin{pgfonlayer}{nodelayer}
		\node [style=black dot] (1) at (0, 0) {};
		\node [style=black dot] (2) at (0, 2) {};
		\node [style=black dot] (3) at (2, 2) {};
		\node [style=none] (4) at (1, 2.25) {$E_Q$};
		\node [style=none] (5) at (-0.27, 1) {$R$};
		\node [style=none] (6) at (1.4, 0.96) {$Q$};
		\node [style=none] (7) at (0, -0.27) {$x$};
		\node [style=none] (8) at (-0.2, 2.27) {$\exists z$};
		\node [style=none] (9) at (2.25, 2.25) {$y$};
	\end{pgfonlayer}
	\begin{pgfonlayer}{edgelayer}
		\draw [style=dashed Latex arrow] (1) to (2);
		\draw [style=Latex arrow] (1) to (3);
		\draw [style=none dashed] (2) to (3);
	\end{pgfonlayer}
\end{tikzpicture}
\end{center}
\end{figure}
\end{definition}

$\mipc$-frames provide a relational semantics for $\mipc$ that extends the usual Kripke semantics for $\ipc$ (see, e.g., \cite[Sec.~10.2]{GKWZ03}). Let $\F=(X,R,Q)$ be an $\mipc$-frame.
A \textit{valuation} on $\F$ is a map associating an $R$-upset to each propositional letter. To see how $\forall$ and $\exists$ are being interpreted in $\F$, let $x \in X$ and $v$ be a valuation on $\F$. Then, for each formula $\varphi$ of $\Lae$,
\begin{align*}
x \vDash_v \forall \varphi & \iff (\forall y \in X) (x Q y \Rightarrow y \vDash_v \varphi);\\
x \vDash_v \exists \varphi & \iff (\exists y \in X) (x E_Q y \; \& \; y \vDash_v \varphi).
\end{align*}

In \cite{Bul65} it was proved that $\mipc$ has the finite model property (fmp for short). The proof contained a gap, which was corrected in \cite{Ono77} and \cite{FS78a}. As a consequence, we obtain:

\needspace{2\baselineskip}
\begin{theorem}
$\mipc$ is Kripke complete. 
\end{theorem}

Since not every extension of $\mipc$ is Kripke complete, we require a more general semantics using descriptive frames. Let $X$ be a topological space. We recall that a subset $U$ of $X$ is {\em clopen} if it is both closed and open, and that $X$ is {\em zero-dimensional} if clopen sets form a basis for $X$. We also recall that $X$ is a {\em Stone space} if it is compact, Hausdorff, and zero-dimensional. A binary relation $R$ on $X$ is {\em continuous} provided the image $R[x] := \{ y \in X : xRy \}$ is closed for each $x\in X$ and the inverse image $R^{-1}[U] := \{ y\in X : yRx$ for some $x\in U\}$ is clopen for each clopen $U$ of $X$.

\begin{definition}\label{def:descriptive mipc-frame}
An $\mipc$-frame $\mathfrak F=(X,R,Q)$ is a \textit{descriptive $\mipc$-frame} if $X$ is equipped with a Stone topology such that
\begin{enumerate}
\item\label{def:descriptive mipc-frame:item2} $R$ is a continuous partial order;
\item\label{def:descriptive mipc-frame:item3} $Q$ is a continuous quasi-order;
\item\label{def:descriptive mipc-frame:item4} 
if $U$ is a clopen $R$-upset, then $Q[U]$ is a clopen $R$-upset.
\end{enumerate}
\end{definition}

\begin{definition}\label{def:mipcfrm-morphisms}
Let $\F_1=(X_1,R_1,Q_1)$ and $\F_2=(X_2,R_2,Q_2)$ be descriptive $\mipc$-frames. A map $f \colon X_1 \to X_2$ is a \emph{morphism of descriptive $\mipc$-frames} if 
\begin{enumerate}
\item\label{def:mipcfrm-morphisms:item1} $f$ is continuous;
\item\label{def:mipcfrm-morphisms:item2} $R_2[f(x)]=fR_1[x]$ for each $x \in X_1$; 
\item\label{def:mipcfrm-morphisms:item3} $Q_2[f(x)]=fQ_1[x]$ for each $x \in X_1$; 
\item\label{def:mipcfrm-morphisms:item4} $Q_2^{-1}[f(x)]= R_2^{-1}fQ_1^{-1}[x]$ for each $x\in X_1$.
\end{enumerate}
\end{definition}

\begin{remark}\label{rem:morph dmsfrm:item1}
In other words, parts \eqref{def:mipcfrm-morphisms:item2} and \eqref{def:mipcfrm-morphisms:item3} of Definition~\ref{def:mipcfrm-morphisms} say that $f$ is a p-morphism with respect to both $R$ and $Q$, while part \eqref{def:mipcfrm-morphisms:item4} is weaker than saying that $f$ is a p-morphism with respect to the inverse of $Q$.
\end{remark}

\begin{definition}
    Let $\mipcfrm$ be the category of descriptive $\mipc$-frames and their morphisms.
\end{definition}

\begin{remark}\label{rem:morph dmsfrm:item2}
It is straightforward to see that isomorphisms in $\mipcfrm$ are homeomorphisms that preserve and reflect $R$ and $Q$.  
\end{remark}

The algebraic semantics for $\mipc$ is given by monadic Heyting algebras \cite{MV57}. Since descriptive $\mipc$-frames are exactly the duals of monadic Heyting algebras \cite[Thm.~17]{Bez99}, by interpreting formulas as clopen $R$-upsets of descriptive $\mipc$-frames, we obtain:

\begin{theorem}\label{thm:completeness descriptive}
Each extension of $\mipc$ is complete with respect to its class of descriptive $\mipc$-frames.
\end{theorem}

\begin{remark}
    Because of Theorem~\ref{thm:completeness descriptive} we mainly work with descriptive $\mipc$-frames, although most of our results can also be formulated in the language of monadic Heyting algebras.
\end{remark}

We next turn our attention to $\ms$. Let $\mathcal{L}_{\Box \forall}$ be a propositional bimodal language with two modalities $\Box$ and $\forall$. As usual, $\Diamond$ abbreviates $\neg\Box\neg$ and $\exists$ abbreviates $\neg\forall\neg$.

\needspace{2\baselineskip}
\begin{definition}\label{def:ms4}
\hfill\begin{enumerate}
\item The \emph{monadic $\sfour$}, denoted $\ms$, is the smallest set of formulas of 
$\mathcal{L}_{\Box \forall}$ containing all theorems of the classical propositional calculus $\cpc$, the $\sfour$-axioms for $\Box$, the $\sfive$-axioms for $\forall$, the left commutativity axiom
\[
\Box \forall p \to \forall \Box p,
\]
and closed under modus ponens, substitution, $\Box$-necessitation, and $\forall$-necessi\-tation.
\item An \emph{extension of $\ms$} is a set of formulas of $\mathcal{L}_{\Box \forall}$ containing $\ms$ and closed under the above rules of inference. 
\end{enumerate}
\end{definition}

\begin{remark}\label{rem:master modality}
We let $\bbox$ denote the compound modality $\Box\forall$. It is immediate from the definition of $\ms$ that $\bbox$ is an $\mathsf{S4}$-modality and that both $\bbox p \to \Box p$ and $\bbox p \to \forall p$ are provable in $\ms$. Therefore, $\bbox$ is a \emph{master modality} for $\ms$ (see, e.g., \cite[p.~71]{Kra99}).
\end{remark}

We will mainly be interested in the following extension of $\ms$.

\begin{definition}
The \emph{monadic Grzegorczyk logic}, denoted $\mgrz$, is the smallest extension of $\ms$ containing the {\em Grzegorczyk axiom}
\[
\grz = \Box(\Box(p\to\Box p)\to p)\to p.
\]
\end{definition}

Kripke semantics for extensions of $\ms$ was introduced by Esakia \cite{Esa88}. To avoid confusion, we write $\F=(X,R,Q)$ for an $\mipc$-frame and $\G=(Y,R,E)$ for an $\ms$-frame.

\begin{definition}\label{def:ms-frame}
An \emph{$\ms$-frame} is a tuple $\G=(Y,R,E)$ such that
\begin{enumerate}
\item\label{def:ms-frame:item2} $R$ is a quasi-order on $Y$;
\item\label{def:ms-frame:item3} $E$ is an equivalence relation on $Y$;
\item\label{def:ms-frame:item4}
$x E y \; \& \; y R z \Longrightarrow \exists u \in Y : x R u \; \& \; u E z$.
\end{enumerate}
\begin{figure}[!ht]
\begin{center}
\begin{tikzpicture}
\node at (-0.2,-0.25) {$x$};
\node at (-0.25,2.27) {$\exists u$};
\node at (2.2,-0.25) {$y$};
\node at (2.2,2.25) {$z$};
\fill (0,0) circle(2pt);
\fill (0,2) circle(2pt);
\fill (2,0) circle(2pt);
\fill (2,2) circle(2pt);
\draw [dashed, -{Latex[width=1mm]}] (0,0) -- (0,2);
\draw [-{Latex[width=1mm]}] (2,0) -- (2,2);
\draw [dashed] (0,2) -- (2,2);
\draw (0,0) -- (2,0);
\node [above] at (1,2) {$E$};
\node [above] at (1,0) {$E$};
\node [left] at (0,1) {$R$};
\node [right] at (2,1) {$R$};
\end{tikzpicture}
\end{center}
\end{figure}
\end{definition}

Kripke semantics for $\sfour$ naturally extends to a relational semantics for $\ms$ by interpreting the modality $\forall$ in $\ms$-frames via the equivalence relation $E$. 
A \textit{valuation} on an $\ms$-frame $\G=(Y,R,E)$ is a map associating a subset of $Y$ to each propositional letter. Then, for each $x \in Y$ and formula $\varphi$ of $\mathcal{L}_{\Box \forall}$, we have
\begin{align*}
x \vDash_v \Box \varphi & \iff (\forall y \in Y) (x R y \Rightarrow y \vDash_v \varphi);\\
x \vDash_v \forall \varphi & \iff (\forall y \in Y) (x E y \Rightarrow y \vDash_v \varphi).
\end{align*}

Since both $\ms$ and $\mgrz$ have the fmp (for the fmp of $\ms$ see \cite[Sec.~6]{BC23a} and the references therein, and for the fmp of $\mgrz$ see \cite{BK24}), we obtain:

\begin{theorem}
$\ms$ and $\mgrz$ are Kripke complete. 
\end{theorem}

As with extensions of $\mipc$, there are extensions of $\ms$ that are Kripke incomplete. We thus work with descriptive $\ms$-frames.

\begin{definition}\label{def:descriptive ms-frame}
An $\ms$-frame $\mathfrak G=(Y,R,E)$ is a \emph{descriptive $\ms$-frame} if $Y$ is equipped with a Stone topology such that
\begin{enumerate}
\item\label{def:descriptive ms-frame:item2} $R$ is a continuous quasi-order;
\item\label{def:descriptive ms-frame:item3} $E$ is a continuous equivalence relation.
\end{enumerate}
\end{definition}

\begin{definition}\label{def:msfrm-morphisms}
Let $\G_1=(Y_1,R_1,E_1)$ and $\G_2=(Y_2,R_2,E_2)$ be descriptive $\ms$-frames. A map $f \colon Y_1 \to Y_2$ is a \emph{morphism of descriptive $\ms$-frames} if  
\begin{enumerate}
\item\label{def:msfrm-morphisms:item1} $f$ is continuous;
\item\label{def:msfrm-morphisms:item2} $R_2[f(x)]=fR_1[x]$ for each $x \in Y_1$;
\item\label{def:msfrm-morphisms:item3} $E_2[f(x)]=fE_1[x]$ for each $x \in Y_1$.
\end{enumerate}
\end{definition}

\begin{remark}
Parts \eqref{def:msfrm-morphisms:item2} and \eqref{def:msfrm-morphisms:item3} of Definition~\ref{def:msfrm-morphisms} say that $f$ is a p-morphism with respect to both $R$ and $E$.
\end{remark}

\begin{definition}
    Let $\msfrm$ be the category of descriptive $\ms$-frames and their morphisms.
\end{definition}

\begin{remark}
It is straightforward to see that isomorphisms in $\msfrm$ are homeomorphisms that preserve and reflect $R$ and $E$.  
\end{remark}

\begin{convention}\label{conv:EcircR}
For an $\ms$-frame $(Y,R,E)$, we denote by $\Q$ the composition $E \circ R$ given by
\begin{equation*}
x \Q y \iff \exists z \in Y : x R z \text{ and } z E y.
\end{equation*}
\end{convention}

\begin{remark}\label{rem:MIPC and MS4 frames}
We briefly compare $\mipc$-frames and $\msfour$-frames. 
In an $\mipc$-frame $(X,R,Q)$ it is the quasi-order $Q$ that is primary and the equivalence relation $E_Q$ is defined from $Q$. On the other hand, in an $\msfour$-frame $(Y,R,E)$ it is the equivalence relation $E$ that is primary and the quasi-order $Q_E$ is defined from $E$ and $R$. 
Nevertheless, there is a close connection between $\mipc$-frames and $\ms$-frames.
Indeed, it follows from \cite[Sec.~2]{Bez99} that if $(X,R,Q)$ is an $\mipc$-frame, then $(X,R,E_Q)$ is an $\ms$-frame. Conversely, if $(Y,R,E)$ is an $\ms$-frame such that $R$ is a partial order,  
then $(Y,R,\Q)$ is an $\mipc$-frame such that $E \subseteq E_{\Q}$, but in general $E \ne E_{\Q}$ (see \cite[p.~24]{Bez99}). 
Therefore,  
this correspondence restricts to a bijection between $\mipc$-frames and those partially ordered $\ms$-frames in which $E = E_{\Q}$. 
(Note that $Q=Q_{E_Q}$ in every $\mathsf{MIPC}$-frame.)
Since 
every finite partially ordered $\ms$-frame is such (see \cite[Lem.~3(b)]{Bez99}), 
the correspondence further restricts to
a bijection between finite $\mipc$-frames and finite partially ordered $\ms$-frames. 
But this bijection does not extend to an equivalence of the corresponding categories 
(viewed as full subcategories of $\mipcfrm$ and $\msfrm$, respectively) because a morphism between finite $\mipc$-frames is not necessarily a morphism between the corresponding finite $\msfour$-frames. 
Furthermore, the bijection between finite $\mipc$-frames and finite partially ordered $\msfour$-frames does not extend to descriptive frames. Indeed, there is a descriptive $\mipc$-frame $(X,R,Q)$ such that $(X,R,E_Q)$ is not a descriptive $\ms$-frame (see \cite[p.~32]{Bez99}). As we pointed out in the introduction (see also Section~\ref{sec: conclusions}), it is this lack of balance between descriptive $\mipc$-frames and descriptive $\msfour$-frames that will be responsible for the failure of the monadic version of Esakia's theorem. 
\end{remark}

The algebraic semantics for $\ms$ is given by monadic $\sfour$-algebras. It is a consequence of J\'onsson-Tarski duality that descriptive $\ms$-frames are the duals of monadic $\sfour$-algebras
(see, e.g., \cite[Thm.~3.11]{BC24a}). Thus, by interpreting formulas as clopen subsets of descriptive $\ms$-frames, we obtain:

\begin{theorem}\label{thm:completeness descriptive MS4}
Each extension of $\ms$ is complete with respect to its class of descriptive $\ms$-frames.
\end{theorem}

In particular, $\mgrz$ is complete with respect to its descriptive frames, which we next recall.
\begin{definition} (see, e.g., \cite[Def.~1.4.9]{Esa19})
Let $R$ be a quasi-order on a set $X$ and $x \in X$.
\begin{enumerate}
\item We call $x$ \emph{maximal} if $xRy$ implies $x=y$. Let $\max X$ be the set of maximal points of $X$.
\item We call $x$ \emph{quasi-maximal} if $xRy$ implies $yRx$. Let $\qmax X$ be the set of quasi-maximal points of $X$.
\end{enumerate}
\end{definition}

While the descriptive frames we work with have multiple relations, when we talk about maximal or quasi-maximal points, 
we always mean with respect to the relation $R$.
Esakia's characterization of descriptive $\Grz$-frames \cite{Esa79} (see also \cite[Thm.~3.5.6]{Esa19}) yields the following characterization of descriptive $\mgrz$-frames. 

\begin{theorem}\label{thm:dual characterization mgrz}
Let $\G=(Y,R,E)$ be a descriptive $\ms$-frame. Then $\G$ validates $\mgrz$ iff $\qmax U = \max U$ for every clopen $U \subseteq Y$. In particular, a finite $\ms$-frame validates $\mgrz$ iff $R$ is a partial order.
\end{theorem}

We recall (see, e.g., \cite[p.~96]{CZ97}) that the \emph{G\"odel translation} $(-)^t$ of $\ipc$ into $\sf S4$ is defined by
\begin{align*}
& \bot^t = \bot\\
& p^t = \Box p \quad \mbox{for each propositional letter } p \\
& (\varphi \land \psi)^t = \varphi^t \land \psi^t \\
& (\varphi \lor \psi)^t = \varphi^t \lor \psi^t \\
& (\varphi \to \psi)^t = \Box (\neg \varphi^t \lor \psi^t).
\end{align*}

Fischer Servi \cite{FS77} (see also \cite{FS78a}) extended the G\"odel translation to a translation of $\mipc$ into $\ms$ as follows:
\begin{align*}
(\forall \varphi)^t &= \bbox \varphi^t \\
(\exists \varphi)^t &= \exists \varphi^t. 
\end{align*}

As a consequence of the fmp of $\mipc$, we have:

\begin{theorem}\cite{FS77,Esa88}\label{thm: MS4 modal comp of MIPC}\label{thm: MGrz modal comp of MIPC}
For each formula $\varphi$ of $\Lae$, 
\[
\mipc \vdash \varphi \iff \ms \vdash \varphi^t \iff \mgrz \vdash \varphi^t.
\]
\end{theorem}

The notions of a modal companion and the intuitionistic fragment (see, e.g., \cite[Sec.~9.6]{CZ97}) have obvious generalizations 
to the monadic setting:

\begin{definition}
Let $\L$ be an extension of $\mipc$ and $\M$ an extension of $\ms$. We say that $\M$ is a \emph{modal companion} of $\L$ and that $\L$ is the \emph{intuitionistic fragment} of $\M$ provided
\[
\L\vdash\varphi \iff \M\vdash\varphi^t
\]
for every formula $\varphi$ of $\Lae$.
\end{definition}

Using this terminology, Theorem~\ref{thm: MS4 modal comp of MIPC} states that both $\ms$ and $\mgrz$ are modal companions of $\mipc$. Our aim is to show that $\mgrz$ is {\bf not} the greatest modal companion of $\mipc$.

\section{Global Kuroda logic}

In this and next sections we will introduce two extensions of $\ms$, which will be utilized in Section~\ref{sec:failure Esakia} to prove our main result.
For this we will use the monadic version of the well-known Kuroda formula $\forall x \, \neg \neg P(x) \to \neg \neg \forall x \, P(x)$, which plays an important role in negation translations of predicate logics (see, e.g., \cite[Sec.~2.3]{TvD88} and \cite[Sec.~2.12]{GSS09}). 

A semantic criterion of when the monadic version of Kuroda's formula is satisfied in descriptive $\mipc$-frames was developed in \cite{Bez00}. We refer to this condition as the Kuroda principle and show that there are two natural versions of it for descriptive $\ms$-frames, which we term the global and local Kuroda principles. These two principles give rise to two extensions of $\ms$. We provide an axiomatization of both and describe their connection to the monadic Kuroda logic. In this section we concentrate on the global Kuroda principle. The local Kuroda principle will be treated in the next section.

\begin{definition}
Let $\kur := \forall \neg \neg p \to \neg \neg \forall p$ be the \emph{monadic Kuroda formula} and $\Kur:=\mipc+\kur$ the \emph{monadic Kuroda logic}.
\end{definition}

\begin{definition}\label{def:KP}
We say that a descriptive $\mipc$-frame $\mathfrak F=(X,R,Q)$ satisfies the \emph{Kuroda principle} ($\mathsf{KP}$) if 
\[
\forall x\in X\,(x \in \max X \Longrightarrow E_Q[x] \subseteq \max X);
\]
equivalently, $E_Q[\max X]=\max X$.
\end{definition}

\begin{theorem}{\cite[Lem.~37]{Bez00}}\label{thm:Kur EQ[max]=max}
A descriptive $\mipc$-frame validates $\Kur$ iff it satisfies $\KP$.
\end{theorem}

$\KP$ has an obvious generalization to descriptive $\ms$-frames.

\begin{definition}
Let $\G=(Y,R,E)$ be a descriptive $\ms$-frame. We say that $\G$ satisfies the \emph{global Kuroda principle} ($\GKP$) if 
\[
\forall x \in Y\,(x\in\qmax Y \Longrightarrow E[x] \subseteq \qmax Y);
\]
equivalently, $E[\qmax Y] = \qmax Y$.
\end{definition}

We next recall the notion of the skeleton of a descriptive $\ms$-frame $\G$ and show that $\G$ satisfies $\GKP$ exactly when its skeleton satisfies $\KP$. 

\begin{definition}\label{def:skeleton}\cite[p.~439]{BBI23} 
For an $\ms$-frame $\G=(Y,R,E)$, define its \emph{skeleton} $\sk\G=(X,R',Q')$ as follows. Let $X:= Y/E_R$ be the quotient of $Y$ by the equivalence relation $E_R$ on $Y$ induced by $R$ (see Convention~\ref{conv:EQ}), and let $\pi \colon Y \to X$ be the quotient map. Define $R'$ on $X$ by 
\[
\pi(x) R' \pi(y) \iff x R y.
\] 
Also, let $\Q=E \circ R$ (see Convention~\ref{conv:EcircR}), and define $Q'$ on $X$ by 
\[
\pi(x) Q' \pi(y) \iff x \Q y.
\]
\end{definition}

\begin{theorem}\label{thm:skeleton descriptive}
\hfill\begin{enumerate}
\item \cite[Lem.~3.3]{BBI23} If $\G$ is an $\ms$-frame, then $\sk\G$ is an $\mipc$-frame.
\item\label{thm:skeleton descriptive:item1} \cite[Thm.~5.14]{BC24a} If $\G$ is a descriptive $\ms$-frame, then $\sk\G$ equipped with the quotient topology is a descriptive $\mipc$-frame. 
\item\label{thm:skeleton descriptive:item1a} \cite[Lem.~5.15]{BC24a} The assignment $\G\mapsto\sk\G$ extends to a functor $\sk \colon \msfrm \to \mipcfrm$.
\item\label{thm:skeleton descriptive:item2} \cite[Prop.~4.11(1)]{BC23a} Let $\varphi$ be a formula of $\Lae$ and $\G$ a \textup{(}descriptive\textup{)} $\ms$-frame. Then
\[
\sk\G \vDash \varphi  \iff \G \vDash \varphi^t.\footnote{For an equivalent algebraic formulation of this result see \cite[Thm.~5]{FS77}.}
\]
\end{enumerate}
\end{theorem}

We recall that each extension $\sf L$ of $\ipc$ has a least modal companion $\tau\sf L$ (see, e.g., \cite[Cor.~9.58]{CZ97}). We generalize the definition of $\tau$ to the monadic setting.

\begin{definition}
For an extension $\L$ of $\mipc$, let $\tau\L$ be the extension of $\ms$ defined by
\[
\tau\L :=  \ms + \{\varphi^t : \L \vdash \varphi \}.
\] 
\end{definition}

The following is an immediate consequence of Theorem~\ref{thm:skeleton descriptive}\eqref{thm:skeleton descriptive:item2}.

\begin{corollary}\label{cor:validity logics skeleton}
Let $\L$ be an extension of $\mipc$ and $\G$ a \textup{(}descriptive\textup{)} $\ms$-frame. Then 
\[
\sk\G \vDash \L \iff \G \vDash \tau\L.
\]
\end{corollary}

\begin{proposition}\label{prop:tau and sigma on Gamma}
Let $\Gamma$ be a set of formulas of 
$\Lae$. Then 
\[
\tau(\mipc+\Gamma)=\ms+\{\gamma^t : \gamma \in \Gamma\}.
\]
\end{proposition}

\begin{proof}
Let $\G$ be a descriptive $\ms$-frame. Corollary~\ref{cor:validity logics skeleton} and Theorem~\ref{thm:skeleton descriptive}\eqref{thm:skeleton descriptive:item2} imply that
\begin{align*}
\G \vDash \tau(\mipc+\Gamma) \iff \sk\G \vDash \mipc+\Gamma \iff 
\G \vDash \ms+\{\gamma^t : \gamma \in \Gamma\}.
\end{align*}
By Theorem~\ref{thm:completeness descriptive MS4}, every extension of $\ms$ is complete with respect to its class of descriptive $\ms$-frames. Thus,
$\tau(\mipc+\Gamma)=\ms+\{\gamma^t : \gamma \in \Gamma\}$.
\end{proof}

\begin{remark}\label{rem: tauL modal comp}
If $\sf L$ is an extension of $\mipc$, then it remains open whether $\tau\sf L$ is a modal companion of $\sf L$.\footnote{
For similar issues in the predicate case see~\cite[Rem.~2.11.13]{GSS09}.}
The main issue here is the absence of a natural functor from $\mipcfrm$ to $\msfrm$ that would allow to show that the functor $\sk \colon \msfrm \to \mipcfrm$ is essentially surjective (that is, every descriptive $\mipc$-frame is isomorphic to the skeleton of some descriptive $\ms$-frame). This is caused by the discrepancy between descriptive $\mipc$-frames and descriptive $\ms$-frames discussed in Remark~\ref{rem:MIPC and MS4 frames}.    
\end{remark}

Let $\L$ be an extension of $\mipc$. While we don't know whether $\tau\L$ is a modal companion of $\L$, we show that this is indeed the case provided $\L$ is Kripke complete.

\begin{proposition}\label{prop:Kripke complete admit modal comp}
Let $\L$ be a Kripke complete extension of $\mipc$. Then $\tau \L$ is a modal companion of $\L$.
\end{proposition}

\begin{proof}
It follows from the definition of $\tau\L$ that $\L \vdash \varphi$ implies $\tau\L \vdash \varphi^t$ for every formula $\varphi$ of $\Lae$. To prove the reverse implication, suppose that $\L \nvdash \varphi$. Since $\L$ is Kripke complete, there is an $\mipc$-frame $\F=(X,R,Q)$ such that $\F \vDash \L$ and $\F \nvDash \varphi$. By Remark~\ref{rem:MIPC and MS4 frames}, $\G=(X,R,E_Q)$ is an $\ms$-frame such that $\sk\G$ is isomorphic to $\F$. Therefore, $\G \vDash \tau \L$ by Corollary~\ref{cor:validity logics skeleton}, and $\G \nvDash \varphi^t$ by Theorem~\ref{thm:skeleton descriptive}\eqref{thm:skeleton descriptive:item2}. Thus, $\tau\L \nvdash \varphi^t$.
\end{proof}

Returning to $\GKP$, we have:

\begin{lemma}\label{lem:clean for max}
Let $\G=(Y,R,E)$ be a descriptive $\ms$-frame.
\begin{enumerate}
\item\label{lem:clean for max:item1} $E[\qmax Y] = E_{\Q}[\qmax Y]$.
\item\label{lem:clean for max:item2} $\G$ satisfies $\GKP$ iff 
\[
\forall x \in Y\,(x\in\qmax Y \Longrightarrow E_{\Q}[x] \subseteq \qmax Y;
\]
equivalently, $E_{\Q}[\qmax Y] = \qmax Y$.
\end{enumerate}
\end{lemma}

\begin{proof}
\eqref{lem:clean for max:item1}.
The left-to-right inclusion is clear because $E \subseteq E_{\Q}$ (see Remark~\ref{rem:MIPC and MS4 frames}). For the other inclusion, let $x \in \qmax Y$, $y \in Y$, and $xE_{\Q}y$. Then $x \Q y$, so 
there is $z \in Y$ with $xRz$ and $zEy$. Since $x \in \qmax Y$ and $xRz$, 
we have $z \in \qmax Y$. Thus, 
$y \in E[\qmax Y]$.

\eqref{lem:clean for max:item2}. This is immediate from (\ref{lem:clean for max:item1}).
\end{proof}

\begin{lemma}\label{lem: pi max and clean}
Let $\G=(Y,R,E)$ be a descriptive $\ms$-frame,
$\sk\G = (X,R',Q')$, and ${\pi \colon Y \to X}$ the quotient map.
\begin{enumerate}
\item \label{lem: pi max and clean:item1} 
$\pi^{-1}[\max X]=\qmax Y$.
\item \label{lem: pi max and clean:item1a}
$\pi^{-1}[E_{Q'}[A]]=E_{\Q}[\pi^{-1}[A]]$ for each $A \subseteq X$.
\item \label{lem: pi max and clean:item2}
$\pi^{-1}[E_{Q'}[\max X]]=E_{\Q}[\qmax Y]$.
\end{enumerate}
\end{lemma}

\begin{proof}
\eqref{lem: pi max and clean:item1}. Let $y \in \pi^{-1}[\max X]$ and $yRz$ for some $z \in Y$. Then $\pi(y)R'\pi(z)$, and hence $\pi(y)=\pi(z)$ because $\pi(y) \in \max X$. 
Therefore, $z R y$, and hence $y \in \qmax Y$. Conversely, suppose that $y \in \qmax Y$ and $\pi(y) R' \pi(z)$ for some $z \in Y$. Then $yRz$, and so $zRy$ because $y \in \qmax Y$. 
Thus, $y E_R z$, which yields $\pi(y)=\pi(z)$. Consequently, $\pi(y) \in \max X$, and hence $y \in \pi^{-1}[\max X]$.

\eqref{lem: pi max and clean:item1a}. 
By definition of $Q'$, for each $y,z \in Y$,
\[
\pi(y) E_{Q'} \pi(z) \iff y E_{\Q} z.
\]
Therefore, $\pi^{-1}[E_{Q'}[x]]=E_{\Q}[\pi^{-1}[x]]$ for each $x \in X$. The result follows since $\pi^{-1}$, $E_{Q'}$, and $E_{\Q}$ commute with set-theoretic union.

\eqref{lem: pi max and clean:item2}.
Applying \eqref{lem: pi max and clean:item1a} and \eqref{lem: pi max and clean:item1},
\[
\pi^{-1}[E_{Q'}[\max X]]=E_{\Q}[\pi^{-1}[\max X]] = E_{\Q}[\qmax Y].
\]
\end{proof}

\begin{proposition}\label{thm:kurbox frame conditions}
A descriptive $\ms$-frame $\G$ satisfies $\GKP$ iff $\sk\G$ satisfies $\KP$.
\end{proposition}

\begin{proof}
By Lemma~\ref{lem:clean for max}\eqref{lem:clean for max:item2}, $\G$ satisfies $\GKP$ iff $E_{\Q}[\qmax Y] = \qmax Y$. On the other hand, by Definition~\ref{def:KP}, $\sk\G=(X,R', Q')$ satisfies $\KP$ iff $E_{Q'}[\max X] = \max X$. Thus, it is sufficient to show that $E_{\Q}[\qmax Y]=\qmax Y$ iff $E_{Q'}[\max X] = \max X$.
Let $\pi \colon Y \to X$ be the quotient map. We have
\[
E_{\Q}[\qmax Y]=\qmax Y \iff \pi^{-1}[E_{Q'}[\max X]]=\pi^{-1}[\max X] \iff E_{Q'}[\max X] = \max X,
\]
where the first equivalence follows from Lemma~\ref{lem: pi max and clean}
and the second 
holds because $\pi$ is onto.
\end{proof}

\begin{definition}
Let $\GKur=\tau(\Kur)$. We call $\GKur$ the \emph{global Kuroda logic}.
\end{definition}

Since $\Kur=\mipc +\kur$, as an immediate consequence of Proposition~\ref{prop:tau and sigma on Gamma} we obtain:

\begin{proposition}\label{prop:GKur ms+kurt}
$\GKur=\ms + \kur^t$.
\end{proposition}

The following theorem shows that $\GKur$ is semantically characterized by $\GKP$.

\begin{theorem}\label{thm:GKur and global Kuroda principle}
A descriptive $\ms$-frame validates $\GKur$ iff it satisfies $\GKP$.
\end{theorem}

\begin{proof}
Let $\G$ be a descriptive $\ms$-frame. Since $\GKur = \tau (\Kur)$, Corollary~\ref{cor:validity logics skeleton} 
yields that 
\[
\sk\G \vDash \Kur \iff \G \vDash \GKur.
\]
Theorem~\ref{thm:Kur EQ[max]=max} implies that $\sk\G \vDash \Kur$ iff $\sk\G$ satisfies $\KP$. 
By Proposition~\ref{thm:kurbox frame conditions},  $\sk\G$ satisfies $\KP$ iff $\G$ satisfies $\GKP$. Thus, $\G \vDash\GKur$ iff it satisfies $\GKP$. 
\end{proof}

\begin{remark}\label{rem:GKur canonical}
Since $\GKP$ is a purely order-theoretic condition that does not involve any topology, the class of monadic $\sfour$-algebras validating $\GKur$ is closed under taking canonical extensions. 
It follows that $\GKur$ is canonical, and hence Kripke complete (see, e.g., \cite[p.~135]{CZ97}). 
It is also natural to investigate whether $\GKur$ has the fmp
and hence is decidable.
Since this is less important for our current purposes, 
we leave it to future work.
\end{remark}

Since $\Kur$ is Kripke complete (see \cite[p.~250]{Bez00}), as an immediate consequence of Proposition~\ref{prop:Kripke complete admit modal comp} 
we obtain:

\begin{theorem}\label{thm:GKur not modal comp Kur}
$\GKur$ is a modal companion of $\Kur$.
\end{theorem}

We end the section by providing a simple axiomatization of $\GKur$.
We abbreviate $\neg \bbox \neg$ by $\bdia$ and point out that $\ms \vdash \bdia p \leftrightarrow \Diamond \exists p$.

\begin{proposition}\label{prop:axiom GKur}
$\GKur = \ms + \bbox \Diamond \Box p \to \Diamond \bbox p = \ms + \Box \bdia p \to \bdia \Box \Diamond p$.
\end{proposition}

\begin{proof}
It is straightforward to check that 
\[
\ms \vdash \bbox \Diamond \Box p \to \Diamond \bbox p \iff \ms \vdash \Box \bdia p \to \bdia \Box \Diamond p.
\]
Therefore, we only show that $\GKur = \ms + \bbox \Diamond \Box p \to \Diamond \bbox p$.
Since
$\GKur = \ms + \kur^t$ (see Proposition~\ref{prop:GKur ms+kurt}),
it is sufficient to show that $\ms + \kur^t = \ms + \bbox \Diamond \Box p \to \Diamond \Box \forall p$. We have
\[
\kur^t=\Box (\bbox\Box \neg \Box \neg \Box p \to \Box\neg \Box\neg \bbox \Box p)=\Box (\bbox \Box \Diamond \Box p \to \Box\Diamond \bbox \Box p).
\]
By necessitation, 
\[
\ms \vdash \Box (\bbox \Box \Diamond \Box p \to \Box\Diamond \bbox \Box p) \iff \ms \vdash \bbox \Box \Diamond \Box p \to \Box\Diamond \bbox \Box p.
\]
Since $\Box$ is an $\sfour$-modality and $\bbox$ is a master modality for $\ms$ (see Remark~\ref{rem:master modality}), 
we have $\ms \vdash \bbox \Box p \leftrightarrow \bbox p$ and $\ms \vdash \Box \bbox p \leftrightarrow \bbox p$.
Thus, using equivalent replacement (see \cite[Thm. 3.65]{CZ97}),
\begin{align*}
\ms \vdash \bbox \Box \Diamond \Box p \to \Box\Diamond \bbox \Box p & \iff \ms \vdash \bbox \Diamond \Box p \to \Box\Diamond \bbox p\\ 
& \iff \ms \vdash \bbox \Diamond \Box p \to \Diamond \bbox p,
\end{align*}
where the last equivalence holds because $\msfour \vdash \bbox \varphi \to \Box\psi$ iff $\msfour \vdash \bbox \varphi \to \psi$ for any pair of formulas $\varphi,\psi$.
Consequently, 
$\ms + \kur^t = \ms + \bbox \Diamond \Box p \to \Diamond \bbox p$.
\end{proof}

\section{Local Kuroda logic}

In this section we introduce the local
Kuroda principle and the corresponding logic, which will play a fundamental role in the proof of the failure of Esakia's theorem in the monadic setting.

\begin{definition}
Let $\G=(Y,R,E)$ be a descriptive $\ms$-frame.
We say that $\G$ satisfies the \emph{local Kuroda principle} ($\LKP$) if 
\[
\forall x \in Y\,(x\in\qmax Y \Longrightarrow \exists y \in E_R[x] : E[y] \subseteq \qmax Y).
\]
\end{definition}

\begin{remark}
While the global Kuroda principle requires that the $E$-equivalence class of each quasi-maximal point is inside $\qmax Y$, the local Kuroda principle asks that this only holds \emph{locally}, meaning 
that each $E_R$-equivalence class of a quasi-maximal point should contain a point whose $E$-equivalence class is inside
$\qmax Y$. In the next proposition we show that the global Kuroda principle is stronger than its local version.
\end{remark}

\begin{proposition}\label{prop:global implies local}
$\GKP$ is strictly stronger than $\LKP$.
\end{proposition}

\begin{proof}
It is straightforward to see that $\GKP$ implies $\LKP$.
We describe a finite $\ms$-frame satisfying $\LKP$ but not $\GKP$.
Let $\mathfrak{H}=(Y,R,E)$ be the three-element frame depicted in Figure~\ref{fig:frame H}, where $R[a]=Y$, $R[b]=R[c]=\{b,c\} = \qmax Y$, $E[c]=\{c\}$, and $E[a]=E[b]=\{a,b\}$. 

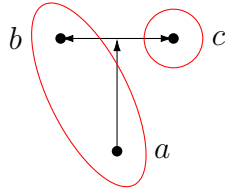
\begin{figure}[h]
\begin{tikzpicture}[-{Latex[width=1mm]}]
\coordinate (B) at (0.75,0);
\coordinate (T) at (0.75,1.5);
\coordinate (TL) at (0,1.5);
\coordinate (TR) at (1.5,1.5);
\fill (B) circle(2pt);
\fill (TL) circle(2pt);
\fill (TR) circle(2pt);
\draw (B) -- (T);
\draw[{Latex[width=1mm]}-{Latex[width=1mm]}] (TL) -- (TR);
\clustertwo{B}{TL}{1.6}{1};
\clusterone{TR}{1.3};
\node at ([shift={(0:0.6)}]B) {$a$};
\node at ([shift={(180:0.6)}]TL) {$b$};
\node at ([shift={(0:0.6)}]TR) {$c$};
\end{tikzpicture}
\caption{The frame $\mathfrak{H}$}\label{fig:frame H}
\end{figure}

It is straightforward to check that $\mathfrak{H}$ 
is an 
$\ms$-frame. Because it is finite, it is also a descriptive $\ms$-frame (whose topology is discrete). Since $a \in E[b]$, $b \in \qmax Y$, and $a \notin \qmax Y$, we see that $\mathfrak{H}$ does not satisfy $\GKP$. 
On the other hand, $c \in E_R[b]$ and $E[c] = \{ c \} \subseteq \qmax Y$. Thus, $\mathfrak{H}$ satisfies $\LKP$.
\end{proof}

\begin{remark}
\hfill\begin{enumerate}
\item
$\KP$ is equivalent to the condition that $E_Q[x]$ is {\em clean} for each  $x \in \max X$, meaning that
$y,z \in E_Q[x]$ and $yRz$ imply $y=z$ (see \cite[Def.~3.6]{BBI23}). 

\item
Similar equivalent conditions exist for both $\GKP$ and $\LKP$. 
Call $E[x]$ in a descriptive $\ms$-frame \emph{quasi-clean} if $y,z \in E[x]$ and $yRz$ imply $zRy$. 
Then $\GKP$
is equivalent to the condition that $E[x]$ is quasi-clean for each $x \in \qmax Y$. On the other hand, $\LKP$
is equivalent to the requirement that for each $x \in \qmax Y$ there is $y \in E_R[x]$ such that $E[y]$ is quasi-clean. 
\end{enumerate}
\end{remark}

We next introduce the logic that is semantically characterized by $\LKP$.

\begin{definition}
Let 
$
\LKur = \ms + \bbox \Diamond \Box p \to \Diamond \forall p.
$
We call $\LKur$ the {\em local Kuroda logic}.
\end{definition}

\begin{remark}
It is straightforward to see that $\LKur$ can be equivalently defined as
\[
\ms + \Box \exists p \to \bdia \Box \Diamond p.
\]
\end{remark}

\begin{lemma}\label{lem: LKur contained in GKur}
$\LKur \subseteq \GKur$.
\end{lemma}

\begin{proof}
Since
$\ms \vdash \Diamond \bbox p \to \Diamond \forall p$, we have
\[
\ms + \bbox \Diamond \Box p \to \Diamond \bbox p \vdash \bbox \Diamond \Box p \to \Diamond \forall p.
\]
Thus,
$\LKur \subseteq \GKur$ by Proposition~\ref{prop:axiom GKur}.
\end{proof}

To see that $\LKP$ gives a semantic characterization of $\LKur$, we recall: 

\begin{definition}
Let $\G=(Y,R,E)$ be an $\ms$-frame and $A \subseteq Y$. We let
\begin{align*}
\Diamond A & := R^{-1}[A] & \Box A & :=  \{ x\in Y : R[x] \subseteq A\}\\
\exists A & := E[A] & \forall A & :=  \{ x\in Y : E[x] \subseteq A\}\\
\bdia A & := \Q^{-1}[A] & \bbox A & :=  \{ x\in Y : \Q[x] \subseteq A\}
\end{align*}
\end{definition}

For a valuation $v$ on $\G$ and a formula $\varphi$ of $\mathcal{L}_{\Box \forall}$, we let $v(\varphi) = \{y \in Y : y \vDash_v \varphi\}$. The following is immediate.

\begin{proposition}
Let $\G$ be an $\ms$-frame, $v$ a valuation on $\G = (Y,R,E)$, and $\varphi$ a formula of $\mathcal{L}_{\Box \forall}$. For every $\bigcirc \in \{\Diamond, \Box, \exists, \forall, \bdia, \bbox\}$ and $x \in Y$,
\[
x \vDash_v \bigcirc \varphi \iff x \in \bigcirc v(\varphi).
\]
\end{proposition}

The following lemma is well known for descriptive $\sfour$-frames (see, e.g., \cite[Sec.~3.2]{Esa19}), and hence it also
holds in descriptive $\ms$-frames.

\begin{lemma}\label{lem:facts about qmax}
Let $\G=(Y,R,E)$ be a descriptive $\ms$-frame.
\begin{enumerate}
\item\label{lem:facts about qmax:item1} $\qmax Y$ is a closed $R$-upset.
\item\label{lem:facts about qmax:item2} \textup{(Fine-Esakia)} For every $x \in Y$ there is $y \in \qmax Y$ such that $xRy$.
\end{enumerate}
\end{lemma}

Recalling that $R$-downsets are complements of $R$-upsets, the following is a consequence of 
\cite[Lem.~3.2.20]{Esa19}. 

\begin{proposition}\label{prop:separation closed upset in descr ms4}
Let $\G=(Y,R,E)$ be a descriptive $\ms$-frame and $U$ a closed $R$-upset of $Y$. 
If $x \notin U$, then there is a clopen $R$-downset $D$ of $Y$ 
such that 
$x \in D$ and $D \cap U = \varnothing$.
\end{proposition}

\begin{theorem}\label{thm:LKur and local Kuroda principle}
A descriptive $\ms$-frame validates $\LKur$ iff it satisfies $\LKP$.
\end{theorem}

\begin{proof}
Let $\G$ be a descriptive $\ms$-frame. Suppose that $\G$ satisfies $\LKP$.
We show that $\G \vDash \bbox \Diamond \Box p \to \Diamond \forall p$. Let $V$ be a clopen subset of $\G$. Then $\bbox \Diamond \Box V$ consists of those points $x \in Y$ such that for every $y \in \Q[x]$ there is $z \in R[y]$ with $R[z] \subseteq V$. 
In particular, if $x$ is such a point and
$y \in \Q[x] \cap \qmax Y$, then 
from $y \in \qmax Y$ and $yRz$ it follows that $zRy$, and hence $y \in R[z] \subseteq V$. \color{black} By Lemma~\ref{lem:facts about qmax}\eqref{lem:facts about qmax:item2}, there is $q \in R[x] \cap \qmax Y$. Then $xRq$ and $\LKP$
implies that there is $t \in R[q]$ with $E[t] \subseteq \qmax Y$. We show that $E[t] \subseteq V$. Let $s\in E[t]$. Since $xRqRtEs$, we have $s \in \Q[x] \cap \qmax Y$, yielding that $s \in V$.
Therefore, $xRt$ and $E[t] \subseteq V$. Thus, $x \in \Diamond \forall V$, showing that $\bbox \Diamond \Box V \subseteq \Diamond \forall V$ for every clopen subset $V$ of $\G$. Consequently, $\G \vDash \bbox \Diamond \Box p \to \Diamond \forall p$.

In order to prove the converse implication, we establish the following:

\begin{claim}\label{claim:local Kuroda}
Let $\G=(Y,R,E)$ be a descriptive $\ms$-frame. If $\G$ does not satisfy $\LKP$,
then there are a nonempty closed $\Q$-upset $U$ and a clopen $R$-upset $V$ such that $\qmax Y \subseteq V$
and $U \subseteq E[V] \cap E[Y \setminus V]$.
\end{claim}

\begin{proof}[Proof of the claim.]
Suppose that $\G$ does not satisfy $\LKP$.
Then there is $q \in \qmax Y$ such that for every $t \in R[q]$ we have that $E[t] \nsubseteq \qmax Y$. By Lemma~\ref{lem:facts about qmax}\eqref{lem:facts about qmax:item1}, $\qmax Y$ is a closed $R$-upset. Since $E[t] \nsubseteq \qmax Y$, there is $s \in Y$ such that $t E s$ and $s \notin \qmax Y$. Therefore, by Proposition~\ref{prop:separation closed upset in descr ms4},
there is a
clopen $R$-downset $W_t$ such that $s \in W_t$ and $W_t \cap \qmax Y = \varnothing$. Since $t \in E[W_t]$ for every $t \in R[q]$, it follows that $R[q] \subseteq \bigcup\{ E[W_t] : t \in R[q]\}$ and hence $\Q[q]=(E\circ R)[q] \subseteq \bigcup\{ E[W_t] : t \in R[q]\}$. Since $\Q[q]$ is closed, compactness of $Y$ yields $t_1, \dots, t_n \in R[q]$ such that $\Q[q] \subseteq E[W_{t_1}] \cup \dots \cup E[W_{t_n}]$. Therefore,
\[
\Q[q] \subseteq E[W_{t_1} \cup \dots \cup W_{t_n}].
\]
Let $V=Y \setminus (W_{t_1} \cup \dots \cup W_{t_n})$. Then $\Q[q] \subseteq E[Y \setminus V]$. Since the $W_{t_i}$ are clopen $R$-downsets such that $W_{t_i} \cap \qmax Y = \varnothing$, we obtain that $V$ is a clopen $R$-upset containing $\qmax Y$. Thus, $R[q] \subseteq \qmax Y \subseteq V$. Consequently, 
$\Q[q]=(E\circ R)[q] \subseteq E[V]$. Let $U=\Q[q]$. Then $U$ is a nonempty closed $\Q$-upset and $V$ is a clopen $R$-upset such that $\qmax Y \subseteq V$ and $U \subseteq E[V] \cap E[Y \setminus V]$.
\end{proof}

Suppose now that $\G$ does not satisfy $\LKP$.
Then Claim~\ref{claim:local Kuroda} implies that 
there are a nonempty closed $\Q$-upset $U$ and a clopen $R$-upset $V$ such that $\qmax Y \subseteq V$
and $U \subseteq E[V] \cap E[Y \setminus V]$. Let $q \in U \cap \qmax Y$, which exists by Lemma~\ref{lem:facts about qmax}\eqref{lem:facts about qmax:item2}. Then every $y \in \Q[q]$ is in $U$ because $U$ is a $\Q$-upset. By Lemma~\ref{lem:facts about qmax}\eqref{lem:facts about qmax:item2}, there is $z \in R[y] \cap U \cap \qmax Y$. Since $U \cap \qmax Y$ is an $R$-upset, we obtain that $R[z] \subseteq U \cap \qmax Y \subseteq V$. Therefore, $q \in \bbox \Diamond \Box V$. However, if $y \in R[q]$, then $y \in U \subseteq E[Y \setminus V]$. Thus, there is $z \in Y \setminus V$ such that $z \in E[y]$. It follows that $q \notin \Diamond \forall V$, showing that $\G \nvDash \bbox \Diamond \Box p \to \Diamond \forall p$.
\end{proof}

\begin{remark}\label{rem:LKur canonical}
$\LKP$ is a purely order-theoretic condition that does not involve topology.
Thus, similarly to $\GKur$ (see Remark~\ref{rem:GKur canonical}), we have that $\LKur$ is canonical, and hence Kripke complete. We leave the investigation of the fmp and decidability of $\LKur$ to future work.
\end{remark}

\begin{proposition}\label{prop:comparison LKur and GKur}
\hfill\begin{enumerate}
\item\label{prop:comparison LKur and GKur:item1} $\LKur$ is strictly contained in $\GKur$.
\item\label{prop:comparison LKur and GKur:item2} $\mgrz \vee \LKur = \mgrz \vee \GKur$.  
\end{enumerate}
\end{proposition}

\begin{proof}
\eqref{prop:comparison LKur and GKur:item1}.
By Lemma~\ref{lem: LKur contained in GKur}, 
$\LKur \subseteq \GKur$.
As observed in Proposition~\ref{prop:global implies local}, there is a descriptive $\ms$-frame $\mathfrak{H}$ that satisfies $\LKP$
but not $\GKP$.
Thus, $\mathfrak{H} \vDash \LKur$ by Theorem~\ref{thm:LKur and local Kuroda principle},
but $\mathfrak{H} \nvDash \GKur$ by Theorem~\ref{thm:GKur and global Kuroda principle}. We conclude that $\GKur \nsubseteq \LKur$.

\eqref{prop:comparison LKur and GKur:item2}. Let $\G=(Y,R,E)$ be a descriptive $\ms$-frame that validates $\mgrz \vee \LKur$. 
By Theorem~\ref{thm:dual characterization mgrz}, $\qmax Y = \max Y$.
By Theorem~\ref{thm:LKur and local Kuroda principle}, $\G$ satisfies $\LKP$,
and so for every $y \in \max Y$ there is $z \in E_R[y]$ with $E[z] \subseteq \max Y$. 
Since $E_R[y] = \{y\}$,
we obtain that $E[y] \subseteq \max Y$ for every $y \in \max Y$, and hence $\G$ satisfies $\GKP$.
It then follows from Theorem~\ref{thm:GKur and global Kuroda principle} that $\G\vDash\GKur$. This shows that $\GKur \subseteq \mgrz \vee \LKur$. Since $\LKur \subseteq \GKur$ by Lemma~\ref{lem: LKur contained in GKur}, we conclude that $\mgrz \vee \LKur = \mgrz \vee \GKur$.
\end{proof}

An alternative characterization of $\Kur$ is given in \cite{Bez00} utilizing morphisms between descriptive $\mipc$-frames. 
Let $\mathfrak{K} = (Y,R,E)$ be the two-element 
$\ms$-frame depicted in Figure~\ref{fig:K}, where 
$Y=\{ a, b \}$, $R[a]=Y$, $R[b]=\{b\}$, and $E[a]=E[b]=Y$.
By \cite[Thm.~43(a)]{Bez00}, 
a descriptive $\mipc$-frame $\F$ validates $\Kur$ 
iff there is no $\mipcfrm$-morphism from a closed $Q$-upset of $\F$ onto the skeleton $\sk\mathfrak{K}$. 
We show that a similar characterization holds for $\LKur$. 
\begin{figure}[h!]
\begin{tikzpicture}[-{Latex[width=1mm]}]
\coordinate (B) at (0,0);
\coordinate (T) at (0,1.5);
\fill (B) circle(2pt);
\fill (T) circle(2pt);
\draw (B) -- (T);
\clustertwo{B}{T}{1.6}{1};
\node at ([shift={(0:0.6)}]B) {$a$};
\node at ([shift={(0:0.6)}]T) {$b$};
\end{tikzpicture}
\caption{The frame $\mathfrak{K}$}\label{fig:K}
\end{figure}
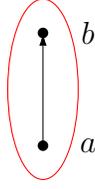

\begin{theorem}\label{thm: LKur iff K notin HS}
Let $\G = (Y,R,E)$ be a descriptive $\ms$-frame. Then $\G \vDash \LKur$ iff there is no $\msfrm$-morphism from a closed $\Q$-upset of $\G$ onto $\mathfrak{K}$. 
\end{theorem}

\begin{proof}
To prove the left-to-right implication, since taking closed $\Q$-upsets and images of onto $\msfrm$-morphisms preserves validity of formulas,  it is sufficient to show that the frame $\mathfrak{K}$ depicted in Figure~\ref{fig:K} does not validate the formula $\bbox \Diamond \Box p \to \Diamond \forall p$ axiomatizing $\LKur$ over $\ms$. Let $U=\{b\}$. We have
\[
\bbox \Diamond \Box U = \bbox \Diamond U = \bbox Y = Y \quad \text{and} \quad \Diamond \forall U = \Diamond \varnothing = \varnothing.
\]
Thus, $\bbox \Diamond \Box p \to \Diamond \forall p$ is not valid on
$\mathfrak{K}$.

We prove the contrapositive of the other implication. Suppose that $\G \nvDash \LKur$. By Theorem~\ref{thm:LKur and local Kuroda principle}, $\G$ does not satisfy $\LKP$. Therefore, by Claim~\ref{claim:local Kuroda}, there are a nonempty closed $\Q$-upset $U$ and a clopen $R$-upset $V$ of $\G$ such that $\qmax Y \subseteq V$
and $U \subseteq E[V] \cap E[Y \setminus V]$. We define a map $f \colon U \to \mathfrak{K}$ by
\[
f(x)=
\begin{cases}
b & \text{if } x \in U \cap V,\\
a & \text{otherwise.}
\end{cases}
\]
We show that $f$ is a well-defined $\msfrm$-morphism. 
Since $U$ is a $\Q$-upset, $U$ is $E$-saturated (that is, $E[U]=U$). Therefore,
from $U \ne \varnothing$ 
and $U \subseteq E[V] \cap E[Y \setminus V]$ it follows that $U \cap V$ and $U \setminus V$ are not empty.
Thus, since $V$ is clopen in $Y$, we obtain that $\{U \cap V, U \setminus V\}$ is a clopen partition of $U$. Hence, $f$ is a well-defined continuous onto map. We next show that $f$ is a p-morphism with respect to $R$. That $V$ is an $R$-upset in $Y$ implies that $U \cap V$ is an $R$-upset in $U$. So, $fR[x] \subseteq R[f(x)]$ for every $x \in U$. Since $\qmax Y \subseteq V$,
if $x \in U$, 
Lemma~\ref{lem:facts about qmax}\eqref{lem:facts about qmax:item2} implies that there is $y \in U \cap \qmax Y \subseteq V$ with $xRy$. So, if $f(x)=a$, then there is $y \in U \cap V$ with $xRy$, and hence $f(y)=b$. Therefore, $R[f(x)] \subseteq fR[x]$  for every $x \in U$. Finally, we show that $f$ is a p-morphism with respect to $E$. Since $a E b$ in $\mathfrak{K}$, it follows that $fE[x] \subseteq E[f(x)]$ for every $x \in U$. Because $U$ is a $\Q$-upset, it follows from $U \subseteq E[V] \cap E[Y \setminus V]$ that for every $x \in U$ there are $y \in U \cap V$ and $z \in U \setminus V$ with $y,z \in E[x]$. Therefore, $E[f(x)] \subseteq fE[x]$ for every $x \in U$. This shows that $f$ is a $\msfrm$-morphism. Thus, $f$ is a $\msfrm$-morphism from the closed $\Q$-upset $U$ of $\G$ onto $\mathfrak{K}$.
\end{proof}

\begin{remark}
Theorem~\ref{thm: LKur iff K notin HS} can be phrased using the language of splitting logics (see, e.g., \cite[Sec.~7.7]{Kra99} and \cite[Sec.~2.4]{Wol93}).
Indeed, 
since the algebraic models of $\ms$ form a variety with equationally definable principal congruences (EDPC), it follows from the general considerations of Blok and Pigozzi \cite{BP82} that
splitting logics above $\ms$ are axiomatized by the Jankov-Fine formulas of finite rooted $\ms$-frames, where an $\ms$-frame $(Y,R,E)$ is rooted if there is $y \in Y$ such that $\Q[y]=Y$.  
By \cite[Cor.~9.64]{CZ97}, if $\L$ is a splitting logic above $\ipc$, then 
$\tau\L$ is a splitting logic above $\sfour$. In fact, if $\sf L$ is axiomatized by the Jankov-Fine formula of a finite $\ipc$-frame $\F$, then $\tau\L$ is axiomatized by the Jankov-Fine formula of $\F$ viewed as an $\sfour$-frame. We show that this is no longer true in the monadic setting. 
By \cite[Thm.~43(a)]{Bez00}, $\Kur$ is the splitting logic above $\mipc$ axiomatized by the Jankov-Fine formula of $\sk\mathfrak{K}$.
By Theorem~\ref{thm: LKur iff K notin HS}, $\LKur$ is the splitting logic above $\ms$ axiomatized by the Jankov-Fine formula of $\mathfrak{K}$.
On the other hand, $\tau(\Kur) = \GKur$, which strictly contains $\LKur$ by Proposition~\ref{prop:comparison LKur and GKur}\eqref{prop:comparison LKur and GKur:item1}. 
Thus, in the lattice of extensions of $\ms$, it is the behavior of $\LKur$, rather than $\GKur$, that is similar to that of $\Kur$ in the lattice of extensions of $\mipc$.
\end{remark}

In the following section we will see that, unlike $\GKur$, the logic $\LKur$ is not a modal companion of $\Kur$. In fact, we will show that $\LKur$ is a modal companion of $\mipc$. This observation is at the heart of the proof of the failure of Esakia's theorem for $\mipc$.

\section{Failure of Esakia's theorem for \texorpdfstring{$\mipc$}{MIPC}}\label{sec:failure Esakia}

In this 
section we show that the monadic analogue of Esakia's theorem fails. In fact, we prove a stronger result: not only is $\mgrz$ not the greatest modal companion of $\mipc$, but $\mipc$ has no greatest modal companion at all! 
We do this by establishing that $\LKur$ is a modal companion of $\mipc$. Since $\mgrz\vee\LKur=\mgrz\vee\GKur$ (see Proposition~\ref{prop:comparison LKur and GKur}\eqref{prop:comparison LKur and GKur:item2}) and $\GKur$ is not a modal companion of $\mipc$ (see Theorem~\ref{thm:GKur not modal comp Kur}), the result follows.

To see that $\LKur$ is a modal companion of $\mipc$, we require the following 
lemma which shows that each finite $\mipc$-frame $\F$ can be realized as the skeleton of a $\LKur$-frame $\G$. Intuitively, $\G$ is constructed by adding to $\F$ a copy of its maximal layer so that each maximal element of $\F$ is $E_R$-related to its copy and the set of these new elements is $E$-saturated in $\G$.

\begin{lemma}\label{thm:skeleton M-frame}
If $\F$ is a finite $\mipc$-frame, then there is a finite $\LKur$-frame $\G$ such that $\F \cong \sk\G$. 
\end{lemma}

\begin{proof}
Let $\F =(X,R,Q)$ be a finite $\mipc$-frame. Consider a set $M$ disjoint from $X$ that is in bijective correspondence with $\max X$. Let $g \colon M \to \max X$ be the bijection,
$Y = X \cup M$, and define $f \colon Y \to X$ by
\[
f(x) =
\begin{cases}
x & \text{if } x \in X,\\
g(x) & \text{if } x\in M.
\end{cases}
\]
Denote by $\Rnew$ the binary relation on $Y$ defined by
\[
x \Rnew y \iff f(x) R f(y).
\]
It is straightforward to check that $\Rnew$ is a quasi-order on $Y$ and that $f$ is a p-morphism from $(Y, \Rnew)$ to $(X,R)$. 
Define a binary relation $\Enew$ on $Y$ by 
\[
x \Enew y \iff
\begin{cases}
x,y \in X \text{ and } xE_Qy, \text{ or}\\
x, y \in M \text{ and } g(x)E_Qg(y).
\end{cases}
\]

\begin{claim}\label{claim:G ms frame}
$\G=(Y, \Rnew, \Enew)$ is an $\ms$-frame.
\end{claim}

\begin{proof}[Proof of the claim.]
Since $E_Q$ is an equivalence relation on $X$, we obtain that $\Enew$ is a well-defined equivalence relation on $Y$. 
Therefore, it remains to check Definition~\ref{def:ms-frame}\eqref{def:ms-frame:item4}.
Let $x,y,z \in Y$ with $x \Enew y$ and $y \Rnew z$. We consider cases. First suppose that $x,y,z \in X$. Then $x E_Q y$ and $y R z$. Since, by Remark~\ref{rem:MIPC and MS4 frames}, $(X,R,E_Q)$ is an $\ms$-frame, there exists $u \in X$ such that $xRu$ and $uE_Qz$. Thus, $x\Rnew u$ and $u\Enew z$. 

The next case to consider is when $x,y \in X$ and $z \in M$. Since $y \Rnew z$, we have that $y R g(z)$. Since $(X,R,E_Q)$ is an $\ms$-frame, there is $u \in X$ such that $x R u$ and $u E_Q g(z)$.
Let $v \in \max X$ be such that $u R v$. Then $v E_Q g(z)$ because $(X,R,E_Q)$ is an $\ms$-frame and $g(z) \in \max X$. Since $x R v$ and $g \colon M \to \max X$ is a bijection, there is $m \in M$ ($m=g^{-1}(v)$) such that $x \Rnew m$ and $m \Enew z$ (because $g(m)\mathop{=}v E_Q g(z)$). 

Observe that, since $x\Enew y$, we have $x \in M$ iff $y \in M$. Therefore, the last case to consider is when $x, y \in M$. From $y \Rnew z$ it follows that $f(y) R f(z)$, and hence $f(z)=f(y)$ because $f(y) \in \max X$ (since $y \in M$). Because $f$ is injective on $M$, we obtain that $z=y$. So, $x \Rnew x$ and $x \Enew y\mathop{=}z$. We have thus shown that $\G$ satisfies Definition~\ref{def:ms-frame}\eqref{def:ms-frame:item4}. Consequently, $\G$ is an $\ms$-frame.
\end{proof}

\begin{claim}\label{claim:G validates LKur}
$\G \vDash \LKur$.
\end{claim}

\begin{proof}[Proof of the claim.]
By Theorem~\ref{thm:LKur and local Kuroda principle}, it is sufficient to verify that $\G$ satisfies $\LKP$. 
We first prove that $\qmax Y = \max X \cup M$. Since $f^{-1}[\max X]=\max X \cup M$, it is sufficient to show that $y \in \qmax Y$ iff $f(y) \in \max X$. Suppose that $y \in \qmax Y$ and $f(y) R x$. Since $f$ is onto, there is $z \in Y$ with $f(z)=x$. Therefore, $f(y) R f(z)$, and so $y \Rnew z$. Thus, $z \Rnew y$ because $y$ is quasi-maximal. From $y \Rnew z$ and $z \Rnew y$ it follows that $f(y) R f(z)$ and $f(z) R f(y)$. Since $R$ is a partial order, we conclude that $x=f(z)=f(y)$. This shows that $f(y) \in \max X$. Conversely, suppose that $f(y) \in \max X$ and $y \Rnew z$. Then $f(y) R f(z)$, and so $f(y)=f(z)$ because $f(y)$ is maximal. Therefore, $f(z) R f(y)$, and hence $z\Rnew y$. Thus, $y \in \qmax Y$. This shows that $\qmax Y = \max X \cup M$.

We now prove that $\G$ satisfies $\LKP$.
Let $y \in \qmax Y$. If $y \in M$, then $\Enew[y] \subseteq M \subseteq \qmax Y$ by definition of $\Enew$. Otherwise, $y \in \max X$, and so $y E_{\Rnew}  y' \in M$ and $\Enew[y'] \subseteq M \subseteq \qmax Y$. In either case, there is $z \in Y$ such that $y E_{\Rnew} z$ and $\Enew[z] \subseteq \qmax Y$. Therefore, $\G$ satisfies $\LKP$.
\end{proof}

\begin{claim}\label{claim:skG iso F}
$\sk\G \cong \F$.
\end{claim}

\begin{proof}[Proof of the claim.]
Since $R$ is a partial order and $y \Rnew z$ iff $f(y) R f(z)$ for every $y,z \in Y$, we have that $y E_{\Rnew} z$ iff $f(y)=f(z)$. Recalling that $\Qnew$ is the composition $\Enew \circ \Rnew$, we show that $y \Qnew z$ iff $f(y) Q f(z)$ for every $y,z \in Y$. It follows from the definition of $\Enew$ that $y \Enew z$ implies $f(y) E_Q f(z)$. Therefore, $y \Qnew z$ implies $f(y) Q_{E_Q} f(z)$, and hence $f(y) Q f(z)$ because $\F$ is an $\mipc$-frame. Thus, $y \Qnew z$ implies $f(y) Q f(z)$. Conversely, suppose that $f(y) Q f(z)$. Then there is $x \in X$ such that $f(y) R x$ and $x E_Q f(z)$.
Therefore, $y \Rnew x$ because $f(y)Rx\mathop{=}f(x)$.
Since $x E_Q f(z)$, we either have that $z \in X$ and $x E_Q z$ or that $z \in M$ and $x E_Q g(z)$.
In the former case, $y \Rnew x$ and $x \Enew z$, so $y \Qnew z$. In the latter case, $y \Rnew x$, $x \Enew g(z)$, and $g(z) \Rnew z$. Thus, $y(\Rnew \circ \Enew \circ \Rnew)z$. We have that $\Rnew, \Enew \subseteq \Qnew$ and $\Qnew$ is transitive because $\Qnew = \Enew \circ \Rnew$ and $\G = (Y, \Rnew, \Enew)$ is an $\ms$-frame by Claim~\ref{claim:G ms frame}. Thus, $y \Qnew z$ in the latter case as well.
This proves that $y \Qnew z$ iff $f(y) Q f(z)$. It is then straightforward to see that $f$ induces a map $f' \colon \sk\G \to \F$ sending the equivalence class $E_{\Rnew}[y]$ to $f(y)$. Since $y \Rnew z$ iff $f(y) R f(z)$ and $y \Qnew z$ iff $f(y) Q f(z)$ for every $y,z \in Y$, it follows that $f'$ is a bijection of $\mipc$-frames that preserves and reflects both relations. Therefore, by Remark~\ref{rem:morph dmsfrm:item2}, $f' \colon \sk\G \to \F$ is an isomorphism of $\mipc$-frames.
\end{proof}

The above three claims finish the proof.
\end{proof}

\begin{remark}
    It is open whether Lemma~\ref{thm:skeleton M-frame} generalizes to arbitrary
    descriptive
    $\mipc$-frames. In fact, it remains open whether the functor $\rho \colon \msfrm \to \mipcfrm$ is essentially surjective (see Remark~\ref{rem: tauL modal comp}).
\end{remark}

\begin{remark}\label{rem:K and H and LKur}
Since the $\ms$-frames $\mathfrak{H}$ and $\mathfrak{K}$ shown in Figures~\ref{fig:frame H} and~\ref{fig:K} will play a fundamental role in this section, we point out that the construction employed in the proof of Lemma~\ref{thm:skeleton M-frame} applied to the $\mipc$-frame $\sk\mathfrak{K}$ yields a frame isomorphic to $\mathfrak{H}$. It then follows from Lemma~\ref{thm:skeleton M-frame} that $\mathfrak{H}$ is an $\LKur$-frame and that $\sk\mathfrak{K}$ and $\sk\mathfrak{H}$ are isomorphic $\mipc$-frames.
\end{remark}

We are now ready to prove that the local Kuroda logic is a modal companion of $\mipc$.

\begin{theorem}\label{thm:LKur modal companion mipc}
$\LKur$ is a modal companion of $\mipc$.
\end{theorem}

\begin{proof}
It is sufficient to show that for every formula $\varphi$ of $\Lae$ we have that
$\mipc \nvdash \varphi$ implies $\LKur \nvdash \varphi^t$. Suppose that $\mipc \nvdash \varphi$. Since $\mipc$ has the fmp, there is a finite $\mipc$-frame $\F$ such that $\F \nvDash \varphi$. By Lemma~\ref{thm:skeleton M-frame}, there is an $\LKur$-frame $\G$ such that $\F \cong \sk\G$. Therefore, $\sk\G \nvDash \varphi$, and so $\G \nvDash \varphi^t$ by Theorem~\ref{thm:skeleton descriptive}\eqref{thm:skeleton descriptive:item2}. Thus, $\LKur \nvdash \varphi^t$.
\end{proof}

\begin{theorem}\label{lem:G2 and G3}
\hfill\begin{enumerate}
\item\label{lem:G2 and G3:item1} $\mathfrak{K} \vDash \mgrz$ but $\mathfrak{K} \nvDash \LKur$.
\item\label{lem:G2 and G3:item2} $\mathfrak{H} \vDash \LKur$ but $\mathfrak{H} \nvDash \mgrz$.
\item\label{lem:G2 and G3:item3} $\mgrz$ and $\LKur$ are incomparable.
\end{enumerate}
\end{theorem}

\begin{proof}
\eqref{lem:G2 and G3:item1}. Since $\mathfrak{K}$ is a finite $\ms$-frame in which $R$ is a partial order, $\mathfrak{K} \vDash \mgrz$ by Theorem~\ref{thm:dual characterization mgrz}.
On the other hand, $\mathfrak{K} \nvDash \LKur$ by Theorem~\ref{thm: LKur iff K notin HS}.

\eqref{lem:G2 and G3:item2}. As was observed in Remark~\ref{rem:K and H and LKur}, $\mathfrak{H} \vDash \LKur$. On the other hand, $\mathfrak{H}\nvDash \mgrz$ by Theorem~\ref{thm:dual characterization mgrz} because it contains a quasi-maximal point that is not maximal.

\eqref{lem:G2 and G3:item3}. This is immediate from \eqref{lem:G2 and G3:item1} and \eqref{lem:G2 and G3:item2}.
\end{proof}

By Theorem~\ref{lem:G2 and G3}, $\mgrz$ cannot be the greatest modal companion of $\mipc$ because it is incomparable with $\LKur$, which is also a modal companion of $\mipc$ by Theorem~\ref{thm:LKur modal companion mipc}.
This already implies that  
the natural generalization of Esakia's theorem to the monadic setting does not hold.
We show that even more is true by proving that there is no greatest modal companion of $\mipc$.

\begin{proposition}\label{prop:mgrz vee lkur not a modal companion}
$\mgrz \vee \LKur$ is not a modal companion of $\mipc$.
\end{proposition}

\begin{proof}
By Proposition~\ref{prop:comparison LKur and GKur}\eqref{prop:comparison LKur and GKur:item2}, $\GKur \subseteq \mgrz \vee \LKur$.
By Theorem~\ref{thm:GKur not modal comp Kur}, $\Kur$ is the intuitionistic fragment of $\GKur$. Therefore, the intuitionistic fragment of $\mgrz \vee \LKur$ contains $\Kur$, which is a proper extension of $\mipc$. Thus, $\mgrz \vee \LKur$ is not a modal companion of $\mipc$.
\end{proof}

\begin{theorem}[Failure of Esakia's theorem for $\mipc$]\label{thm:no greatest mod comp}
There is no greatest modal companion of $\mipc$.
\end{theorem}

\begin{proof}
Suppose that there is a greatest modal companion $\M$ of $\mipc$. By Theorems~\ref{thm: MS4 modal comp of MIPC} and~\ref{thm:LKur modal companion mipc}, both $\mgrz$ and $\LKur$ are modal companions of $\mipc$, and hence $\mgrz \vee \LKur \subseteq \M$. Therefore, the intuitionistic fragment of $\mgrz \vee \LKur$ is contained in the intuitionistic fragment of $\M$ which is $\mipc$, and hence $\mgrz \vee \LKur$ is a modal companion of $\mipc$. But this contradicts Proposition~\ref{prop:mgrz vee lkur not a modal companion}. Thus, $\M$ does not exist.
\end{proof}

\begin{remark}
Although there is no greatest modal companion of $\mipc$, a standard argument utilizing Zorn's lemma shows that every modal companion of $\mipc$  is contained in a maximal modal companion of $\mipc$. We leave it as an open problem to determine the cardinality of the set of maximal modal companions of $\mipc$, and whether $\mgrz$ is one of those.
\end{remark}

\section{Conclusions}\label{sec: conclusions}

We have demonstrated that Esakia's theorem that $\Grz$ is the greatest modal companion of $\ipc$ fails as soon as we add to the language monadic quantification of one fixed variable. As we pointed out in the introduction, the reason 
why Esakia's theorem holds in the propositional case could be summarized as follows: the category 
of descriptive $\ipc$-frames is equivalent to the category of partially ordered descriptive $\sfour$-frames and $\Grz$ is complete with respect to the latter category. If $\M$ is a modal companion of $\ipc$, then this category 
is a full subcategory of the category of descriptive $\M$-frames, yielding that $\M$ is contained in $\Grz$.

The situation changes considerably in the monadic setting since the addition of monadic modalities breaks the above correspondence between the two semantics. Indeed, while $Q$ is a continuous relation in a descriptive $\mipc$-frame $\F=(X,R,Q)$, the relation $E_Q$ may no longer be continuous. 
On the other hand, in a descriptive $\msfour$-frame $\G=(Y,R,E)$, the relation $E$ is continuous. In addition, a morphism between descriptive $\mipc$-frames does not have to be a p-morphism with respect to $E_Q$, while a morphism between descriptive $\msfour$-frames must be a p-morphism with respect to $E$. Thus, while the embedding $\sigma\colon \ipcfrm\to\sfrm$ yields an equivalence between $\ipcfrm$ and the category of partially ordered descriptive $\sfour$-frames, its natural extension to the monadic setting is {\em not} even well defined (neither on objects nor on morphisms). 
Consequently, although $\mgrz$ is complete with respect to the category of partially ordered descriptive $\msfour$-frames \cite{BK24}, if $\M$ is a modal companion of $\mipc$, the latter category 
is no longer a full subcategory of the category of descriptive $\M$-frames, resulting in the failure of the monadic version of Esakia's theorem.
One way to remedy this is to add appropriate axioms to restore the semantic balance enjoyed in the case of $\ipc$ and $\Grz$. This will be discussed in a follow up paper.

We conclude the paper by addressing Naumov's claim 
that $\QGrz+\Box \exists xP(x) \to \Diamond \exists x\Box P(x)$ is a modal companion of $\iqc$ that strictly contains $\QGrz$ (see \cite{Nau91} and \cite[Thm. 2.11.14]{GSS09}).
More specifically, we show that $\mgrz + \Box \exists p \to \Diamond \exists \Box p$ is equal to $\mgrz\vee\GKur$, from which we derive that Naumov's logic is not a modal companion of $\iqc$. Therefore, while we verified Naumov's claim that $\QGrz$ is not the greatest modal companion of $\iqc$ for the monadic fragments of these logics, the full predicate case requires further study (see Remark~\ref{rem:discussion predicate} for more details). 

\begin{definition}
Let $\N = \ms + \Box \exists p \to \Diamond \exists \Box p$. We call $\N$ the \emph{Naumov logic}.
\end{definition}

Because $\ms \vdash \bdia p \leftrightarrow \Diamond \exists p$ (see the paragraph before Proposition~\ref{prop:axiom GKur}), the following is straightforward.

\begin{lemma}\label{lem:axiom N}
$\N = \ms + \Box \exists p \to \bdia \Box p = \ms + \bbox \Diamond p \to \Diamond \forall p$.
\end{lemma}

By Proposition~\ref{prop:comparison LKur and GKur}\eqref{prop:comparison LKur and GKur:item2}, $\LKur$ and $\GKur$ coincide over $\mgrz$. The next proposition shows that over $\mgrz$ they also coincide with $\N$.

\begin{proposition}\label{prop:local global Kur and N over grz}
$\mgrz \vee \GKur = \mgrz \vee \LKur = \mgrz \vee \N$.
\end{proposition}

\begin{proof}
It is sufficient to show that $\mgrz \vee \LKur = \mgrz \vee \N$. Since $\LKur = \ms + {\bbox \Diamond \Box p \to \Diamond \forall p}$ and $\N = \ms + \bbox \Diamond p \to \Diamond \forall p$ (see Lemma~\ref{lem:axiom N}), it is enough to show that $\mgrz \vdash \bbox \Diamond \Box p \leftrightarrow \bbox \Diamond p$, which can be seen as follows. 
It is well known that the McKinsey formula $\Box \Diamond p \to \Diamond \Box p$ is a theorem of $\Grz$.
Therefore, $\Grz \vdash  \Box \Diamond \Box p \leftrightarrow \Box \Diamond p$, and so $\mgrz \vdash \bbox \Box \Diamond \Box p \leftrightarrow \bbox \Box \Diamond p$.
Thus, since $\bbox$ is a master modality, $\mgrz \vdash \bbox \Diamond \Box p \leftrightarrow \bbox \Diamond p$ (see the proof of Proposition~\ref{prop:axiom GKur}).
\end{proof}

\begin{remark}
Let $\mathsf{MS4.1} = \ms + \Box \Diamond p \to \Diamond \Box p$. By arguing as in the proof of the previous proposition, 
$\mathsf{MS4.1} \vee \LKur = \mathsf{MS4.1} \vee \N$. It follows from \cite[p.~154]{Esa79b} (see also \cite[Prop.~3.46]{CZ97}) that descriptive $\mathsf{MS4.1}$-frames are exactly those descriptive $\mathsf{MS4}$-frames in which every quasi-maximal point is maximal.
Therefore, by arguing as in the proof of 
Proposition~\ref{prop:comparison LKur and GKur}\eqref{prop:comparison LKur and GKur:item2},
we also have that
$\mathsf{MS4.1} \vee \LKur = \mathsf{MS4.1} \vee \GKur$. Thus, $\GKur$, $\LKur$, and $\N$ are all instances of Kuroda-like logics that collapse over $\mathsf{MS4.1}$. It is natural to investigate such principles in more detail.
\end{remark}

\begin{theorem}\label{thm:Naumov}
$\QGrz+\Box \exists xP(x) \to \Diamond \exists x\Box P(x)$ is not a modal companion of $\iqc$.
\end{theorem}

\begin{proof}
It is clear that the logic $\mgrz \vee \N$ is contained in the monadic fragment of $\QGrz+\Box \exists xP(x) \to \Diamond \exists x\Box P(x)$. It follows from Propositions~\ref{prop:mgrz vee lkur not a modal companion} and~\ref{prop:local global Kur and N over grz} that $\mgrz \vee \N$ is not a modal companion of $\mipc$. Therefore, $\QGrz+\Box \exists xP(x) \to \Diamond \exists x\Box P(x)$ proves the translation of a one-variable formula that is not a theorem of $\iqc$, and hence it cannot be a modal companion of $\iqc$.
\end{proof}

\begin{remark}\label{rem:discussion predicate}
As we pointed out in the introduction, Pankratyev \cite{Pan89} claimed that $\QGrz$ is a modal companion of $\iqc$,
however his proof relied on the Flagg-Friedman translation~\cite{FF86} of $\qsfour$ to $\iqc$, which is not faithful \cite{Ino92}.
Therefore, it remains open whether $\QGrz$ is a modal companion of $\iqc$.
Unfortunately, Kripke completeness of $\iqc$ does not help since it relies on non-noetherian predicate Kripke frames which are not models of $\QGrz$. A possible approach would be to use the more general Kripke bundle semantics \cite[Ch.~5]{GSS09} which would then require to prove that $\iqc$ is complete with respect to noetherian Kripke bundles. As far as we know, this remains an open problem.

If $\QGrz$ ends up being a modal companion of $\iqc$, 
Naumov's claim that $\QGrz$ is not the greatest such
would also require further investigation. As we saw in Theorem~\ref{thm:Naumov}, the logic $\QGrz+\Box \exists xP(x) \to \Diamond \exists x\Box P(x)$ is not going to be useful for this purpose. It is more convenient to consider
\[
\mathsf{QLKur} := \qsfour + \Box \forall x \Diamond \Box P(x) \to \Diamond \forall x P(x),
\]
which is the predicate version of $\LKur$, and is not comparable with $\QGrz$. 
However, this approach requires a proof that $\mathsf{QLKur}$ is a modal companion of $\iqc$. 
By Remark~\ref{rem:LKur canonical}, $\LKur$ is Kripke complete, from which it follows that $\LKur$ axiomatizes the monadic fragment of $\mathsf{QLKur}$.
Therefore, Theorem~\ref{thm:LKur modal companion mipc} yields that the monadic fragment of $\mathsf{QLKur}$ is a modal companion of $\mipc$.
However, it remains open whether $\mathsf{QLKur}$ is a modal companion of $\iqc$.
\end{remark}

\subsection*{Acknowledgements}
We would like to thank the referee for careful reading and useful comments which have improved the presentation.

\end{document}